\documentclass[a4paper,draft]{amsproc}
\usepackage{amssymb}
\usepackage[colorlinks=true]{hyperref}
\usepackage{thmtools}
\usepackage{datetime}
\newdate{date}{26}{06}{2019}
\date{\displaydate{date}}

\usepackage{cite}

\theoremstyle{plain}
 \newtheorem{thm}{\textbf{Theorem}}[section]
 \newtheorem{prop}{\textbf{Proposition}}[section]
 \newtheorem{lem}{\textbf{Lemma}}[section]
 
\theoremstyle{definition}
 
 \newtheorem{dfn}{\textbf{Definition}}[section]
\theoremstyle{remark}
 \newtheorem{rem}{\textbf{Remark}}[section]
 \numberwithin{equation}{section}
 
\usepackage{hyperref}
\usepackage{enumitem}
\usepackage{relsize,exscale}
\usepackage[utf8]{inputenc}
\usepackage{stmaryrd}
\usepackage{dsfont}
\usepackage{amsfonts}
\usepackage{amsmath}
\usepackage{mathtools}
\usepackage[T1]{fontenc}
\usepackage{amssymb}
\DeclareMathAlphabet{\numbermathbb}{U}{bbold}{m}{n}

\renewcommand{\le}{\leqslant}

\title[Refactorisation of the Dirichlet convolution]{Refactorisation of the Dirichlet convolution}

\keywords{unitary, arithmetic, dirichlet, convolution, factorisation,  convolution, riemann hypothesis, zeta, euler product, prime number, L function, analytic number theory, zeta zeros, derivation}

\author[A.E.L]{\bfseries Ansar EL HASSANI} 

\address{ \centerline{ 
Department of Mathematics \\ 
City university \\ 
London\\
United kingdom } }

\address{\centerline{ 
Department of Mathematics \\ 
Ecole centrale marseille \\ 
Marseille\\
France } }

\email{biralimedel@gmail.com}

\begin{document}

\vspace{18mm} \setcounter{page}{1} \thispagestyle{empty}

\begin{abstract}
We present a new way to factor the dirichlet convolution for completely multiplicative functions whitch led us to constructing a ring that arise from the operations involved in the factorisation. We will conclude by some identities that was found during this work.  \newline
An application of the results gives us a generalisation of the following Hardy formula : 
\[ \zeta(x)^{2} = \zeta(2x)\sum_{m=1}^{+\infty} \frac{2^{\omega(m)}}{m^{x}}\]
which is : \newline
\begin{align*}
|\zeta(z)|^{2}
& = \zeta(2x)\sum_{m=1}^{+\infty}\frac{1}{m^{x}}2^{\omega(m)}\prod_{p | m , p \in \mathbb{P}}^{\omega(m)}\cos(y\ln(p^{v_{p}(m)})) \\
\end{align*}
with:  \newline
	x > 1  in Hardy's formula \newline
	z be a complex number with z = x+iy and  $ \Re(z) > 1 $ \newline
	$ \omega(m) $ : number of unique primes in m \newline
	$ v_{p}(m) $ : power of the prime p in m  \newline
	$ \mathbb{P} $ : set of the prime number
\end{abstract}

\maketitle

\centerline{\displaydate{date}}


\tableofcontents

\allowdisplaybreaks

\section{Introduction}
We will begin by writing some definitions that will be used throughout this article, then we will state the results obtained starting with what we see as the most important ones. 
In the last sections, we will prove our claims and discuss possible applications and future investigations over this work.

\begin{dfn}

The zeta series was discovered by Leonhard Euler and is defined as follows:
\[ \forall s \in \mathbb{C} , ~   \Re(s)>1 \: ~ \zeta(s)  = \sum_{n=1}^{+\infty} \frac{1}{n^{s}} \]

\end{dfn}

The series $ \sum_{n=1}^{+\infty} \frac{1}{n^{s}} $ does not converge for $ \Re(s) \le  1 $ , however we can define an extension (unique in the sense of the analytic extension) on the complex plane, except in $ s = 1 $.
\\

This analytic extension is what is called the zeta function \cite{Riemann} \cite{Edwards}  \cite{Titchmars} .

\begin{dfn}

Euler has also established the "Eulerian" product associated with the zeta function:

\[ \forall s \in \mathbb{C} , ~   \Re(s)>1  ~  \zeta(s) = \sum_{n=1}^{+\infty} \frac{1}{n^{s}} = \prod_{p \in \mathbb{P}} \frac{1}{1-\frac{1}{p^{s}}} \]

\end{dfn}

To understand more easily the importance of this formula, it is more appropriate to write it as follows:

\[ \forall s \in \mathbb{C} ,  ~  \Re(s)>1  ~  \prod_{p \in \mathbb{P}} \frac{1}{1-\frac{1}{p^{s}}} = \prod_{p \in \mathbb{P}} 1+\frac{1}{p^{s}}+\frac{1}{p^{2s}}+\frac{1}{p^{3s}}+... \]

\[ \forall s \in \mathbb{C} ,  ~  \Re(s)>1  ~ \zeta(s)  = \sum_{n=1}^{+\infty} \frac{1}{n^{s}} = 1+\frac{1}{2^{s}}+\frac{1}{3^{s}}+\frac{1}{4^{s}}+...  \]

The idea is that the formula $ \sum_{n=1}^{+\infty} \frac{1}{n^{s}} $ presents the integers as a sequence of ordered numbers separated  by a fixed step (additive, metric and ordered vision) , while the formula $ \prod_{p \in \mathbb{P}} \frac{1}{1-\frac{1}{p^{s}}} $ is an integer combination of different powers of prime numbers (multiplicative and combinatory vision). \\

\begin{dfn}

Riemann proved \cite{Riemann} that the zeta function satisfies the following functional equation:

\[ \forall s \in \mathbb{C}- \left\{ 1 \right\} ~ : ~ \zeta(s) = 2^s \pi^{s-1} \sin \left( \frac{\pi s}2\right) \Gamma(1-s)\zeta(1-s) \]

\end{dfn}

\begin{rem}

This result expresses a relation between the values $ \zeta(s) $ and $ \zeta(1-s) $, and more particularly, it indicates that the set of zeros of the zeta function is symmetrical with respect to the line. $ \Re(s) = \frac{1}{2} $.

\end{rem}

The Riemann Hypothesis describes a pattern, a regularity that prime numbers are supposed to follow. This pattern is the fact that they are, at least in part, distributed pseudo randomly. Knowing this information allows us to better frame them.

One of the strategies that mathematicians have applied to try to answer this question is the generalization \cite{Connes} :

They tried to define several classes of functions having the same fundamental properties of the zeta function (Eulerian product, functional equation, Riemann hypothesis ...) and sought to understand these new functions. This allows to focus only on the important properties related to the Riemann hypothesis.
\\

Several classes of functions have been defined, some of these courses are at the heart of a major research program such as Polya-Hilbert program and The Langlands program .

\section{Acknowledgement}
I would like to thank my supervisor Yang Hui He for all the time and the discussions he gave me several times a week for the duration of this research.

I thank my doctoral colleagues for presenting me their work, for explaining me some concepts and for the time we shared together.

I thank my roommates, with whom this trip to London was an excellent experience.

I thank Ibtihaj Sadiki, Mohammed Allali, Nizar Bellazrak, Olivier Roux, Oumaima Sadiki, and especially my father for their many corrections and re-readings.

I thank Olivier Ramaré for his review and his many corrections and comments.

I thank my parents, family and friends for everything.
\section{Definitions}
We will start the definitions by recalling one of the most important family of functions that arise in many of the problems in number theory

\begin{dfn}
	The set of multiplicative function is defined by the following properties :
	\[
	F \in \mathbb{M} \iff 
	\begin{array}{l rcl}
	{F} \textnormal{ : }  \mathbb{N}^{\star}  \longrightarrow  \mathbb{C} \\
	F(1)=1 \\
	(a , b) =1 \Rightarrow F(ab)=F(a)F(b) \\
	\end{array}
	\]

\end{dfn}

We also often encounter a smaller family of functions $ \mathbb{M}_{c} $ yet more simple than the multiplicative fuctions of the set $ \mathbb{M} $

\begin{dfn}
	The set of completely multiplicative function $ \mathbb{M}_{c} $ is describled by : 
	\[
	F \in \mathbb{M}_{c} \iff 
	\begin{array}{l rcl}
	F \in \mathbb{M} \\
	a,b \in \mathbb{N^{\star}}  : F(ab)=F(a)F(b) \\
	\end{array}
	\]
\end{dfn}

\begin{rem}
	The set $ \mathbb{M}_{c} $ is stable by multiplication, which is useful for building structures over it
\end{rem}

$ \mathbb{M}_{c} $ family contains many functions that are not particularly linked with number theory as the constraints defining the set are general
As an example , we could give the exponential $ x \rightarrow \exp(x) $  , and the module $ x \rightarrow |x| $
$ \\ $

Conversely, the set $ \mathbb{M} $ contains many important functions in arithmetic, one of the most useful in analytic number theory being Dirichlet's characters

\begin{rem}
	 In this paper , we will use the following notation $ \\ $
	 $ \forall a , b \in \mathbb{N^{\star}} ~  ~ (a , b) $ : refere to the grand comun divisor of a and b  
\end{rem}

\begin{dfn}
	The omega function $ \omega $ is defined as follow : 
	\[
	\forall n \in \mathbb{N}^{*} ~ \omega(n) = \sum_{p|n} 1
	\]
\end{dfn}

\begin{dfn}
	A Dirichlet character is a multiplicative function that verify :
	\[
	{\chi} \textnormal{ : }  \mathbb{N}  \longrightarrow  \mathbb{C} 
	\]
	\[\forall a,b \in \mathbb{N} \Rightarrow \chi(ab)=\chi(a)\chi(b)\]
	\[\exists k \in \mathbb{N^{\star}} , \forall n \in \mathbb{N} ~ \chi(n) = \chi(n + k) \]
	\[ \forall n \in \mathbb{N^{\star}} ~ (n , k) > 1 \Rightarrow \chi(n) = 0
	\]
\end{dfn}

\begin{rem}
	A Dirichlet character define a multiplicative function.
\end{rem}

\begin{dfn}
	The generative dirichlet serie associated with a multiplicative function $ F $  is defined as :
	\[\forall F \in \mathbb{M} , \forall s \in \mathbb{C} , \Re(s) > \sigma  : ~ D(F,s)=\sum_{n=1}^{\infty} \frac{F(n)}{n^{s}} \]
	where $\sigma $ define a reagion of the complex plan $\mathbb{C}$ where the serie converge.

\end{dfn}

The dirichlet characters define a family of functions that generalized some of the specificities of the rieman zêta function, and thus gave rise to a new more general conjecture than Riemann hypothesis.

\begin{dfn}
	The L series are the Dirichlet generative series associated with a dirichlet character :
	\[\forall s \in \mathbb{C} , \Re(s)>1 : ~ L(\chi,s) = \sum_{n=1}^{\infty} \frac{\chi(n)}{n^{s}} \]
	
\end{dfn}

\begin{dfn}
	The L functions are the analytical extention of the L series
	
\end{dfn}

\begin{rem}
	The Riemann zêta function is the L serie for the caracter $ \mathds{1} $. \[ \forall n \in \mathbb{N} : \mathds{1}(n) = 1  \] 
\end{rem}

\begin{dfn}
	The generalised riemann hypothesis (GRH) is the assertion that the following proposition is true for all L functions :
	
	\[ \forall s \in \mathbb{C} ~,~ 0<\Re(s)<1 : L(\chi,s)=0 \Rightarrow \Re(s)=\frac{1}{2} \]
	
\end{dfn}

\begin{rem}
	The set of complexe number that have $  0<\Re(s)<1   $ is often called the critical band and the ligne of complexe number defined by  $ \Re(s)=\frac{1}{2} $ is called the critical ligne .
\end{rem}

We will now present some operations that has been studied in this work and its associated ring.

\begin{dfn}
	We introduct the weight functions $ W $ as follow :
	$$
	\begin{array}{l rcl}
	W  :   {\mathbb{N}^{\star}}^{2}  \longrightarrow  \mathbb{C} \\
	\end{array}
	$$
	
\end{dfn}

\begin{dfn}
	We define the operation $\, \underset{w}{\scalebox{0.6}{$\square$}} \, $, which we will call the W-convolution :
	$$ \forall F,G \in \mathbb{M} , \forall m \in \mathbb{N}^{\star}  :   [F  \, \underset{w}{\scalebox{0.6}{$\square$}} \, G](m) = \sum_{ab=m} F(a)G(b)W(a,b) $$

\end{dfn} 

\begin{rem}
For notation purpuses , we will write this notation simply  $\, \scalebox{0.6}{$\square$} \, $ .
\end{rem}

\begin{dfn}
	We define the operation (the multiplication) $ \times $ as follow : 
	$$ \forall F,G \in \mathbb{M} , \forall m \in \mathbb{N}^{\star}  :   [F \times G](m) = F(m)G(m) $$ 
	
\end{dfn}

Here we will define what we mean by a derivation over a ring

\begin{dfn}
	A derivation $ \nabla $ over $ \mathbb{M} $ is an operator how act over the set $ \mathbb{M} $ which respect :
	\[ \forall F,G \in \mathbb{M} ~,~ \nabla(F ~ \scalebox{0.6}{$\square$} ~ G) = \nabla(F) ~ \scalebox{0.6}{$\square$} ~ \nabla(G)  \]
	\[ \forall F,G \in \mathbb{M} ~,~ \nabla(F \times G) = \nabla(F) \times G ~ \scalebox{0.6}{$\square$} ~  F \times \nabla(G) \]

\end{dfn}

We will end this reminder of the definitions used here by recalling the Dirichlet ring of arithmetic functions:

\begin{dfn}
	The  Dirichlet ring $ (\Omega, \star ,.) $ of arithmetic functions is constructed as follows:
	\[ F \in \Omega \iff F \textnormal{ : }  \mathbb{N}  \longrightarrow  \mathbb{C} \]
	\[ \forall F,G \in \Omega \textnormal{ : }  [F \star G] (n)= \sum_{ab=n} F(a)G(b)  \]
	\[ \forall F,G \in \Omega \textnormal{ : }  [F . G] (n)= F(n)G(n)  \]
	
\end{dfn}

\begin{rem}
The operation $ \star $ used in the ring is called the Dirichlet convolution
	\[ \forall F,G \in \Omega \textnormal{ : }  [F \star G] (n)= \sum_{ab=n} F(a)G(b)  \]
\end{rem}

\section{Results}
We will present the results that was established throughout this work starting by what we sees as the most original ones.

Using this convolution and multiplication over the set of multiplicative functions $ \mathbb{M} $ , we have constructed the following result 

\begin{restatable}[Ring unicity theorem]{thm}{unicity}\label{thm:unicity}

	\[ (\mathbb{M},  \scalebox{0.6}{$\square$}  , \times ) ~ \, \textnormal{is a commutative ring } \iff 
	\begin{array}{l rcl}
	\forall a,b \in \mathbb{N}^{\star}  :  W(a,b) =
	\begin{cases}
	\begin{array}{ll}
	1 \,\textnormal{ if } \, (a ,  b) = 1 \\
	0 \,\textnormal{ if not}
	\end{array}
	\end{cases}
	\end{array}  \]

\end{restatable}

The operations used to define the ring  $ (\mathbb{M},  \scalebox{0.6}{$\square$}  , \times ) $ could then be used to express the following equalities witch allows us to give a new formula for the dirichlet series

\begin{restatable}[Refactoristion theorem]{thm}{refactoristion}
\label{thm:refactoristion}

	\[\forall F,G \in \mathbb{M}_{c} :  D(F \star G,s) = D(F \times G,2s) \times D( F \square  G ,s) \]
	
\end{restatable}

a possible reformulation could be :

\begin{prop}
\textnormal{for F and G completly multiplicatives :}

	\[\forall F,G \in \mathbb{M}_{c} :  D(F,s) \times D(G,s) = D(F \times G,2s) \times D( F \square  G ,s) \]
	
\end{prop}

We then have the following useful related result

\begin{restatable}{prop}{realimsplit}\label{prop:realimsplit}

\[  
\forall F , G \in \mathbb{M}_{c} : ~ D(F,s) \times D(G,\overline{s}) = D(F \times G,2x) \times D( F \times \text{Id}_{e}^{-iy} \square G\times \text{Id}_{e}^{iy} , x ) \]
	
\end{restatable}

This last proposition could have some direct applications as the $ D(F \times G,2x) $ section does not depend on y ( $s \coloneqq x + iy $ )
\newline
For example , it allows us to generalise one of Hardy's formula  \cite{Hardy} :

\begin{prop}
$\forall z \in \mathbb{C} , \Re(z) > 1 :$
	\begin{align*}
	\zeta(z)\zeta(\bar{z})= |\zeta(z)|^{2}
	& = \zeta(2x)\sum_{m=1}^{+\infty}\frac{1}{m^{x}}2^{\omega(m)}\prod_{j=1}^{\omega(m)}\cos(y\ln(p_{j}^{v_{p_{j}}(m)})) \\
	\end{align*}
	
\textnormal{This formula makes it possible to deduce a known result of Hardy by taking $ y = 0 $ :}
	\[ \zeta(x)^{2} = \zeta(2x)\sum_{m=1}^{+\infty} 
	\frac{2^{\omega(m)}}{m^{x}}\]
	
\end{prop}

The two following indentities  independent was proven aside during this work , they do not make use of the previous unitary arithmetic framework

\begin{restatable}{prop}{zetaminusone}\label{prop:zetaminusone}

	\[
	\forall s \in \mathbb{C} , \Re(s)>1 ~ : ~
	\sum_{n=1}^{\infty} \frac{1}{n^{s}} \cdot \sum_{p \in \mathbb{P}} \frac{1}{p^{s}} 
	\frac{1}{\omega(np)}  = \zeta(s)-1
	\]

\end{restatable}

\begin{restatable}{prop}{zetaminusonenext}\label{prop:zetaminusonenext}

	\[
	\forall s \in \mathbb{C} , \Re(s)>1 ~ : ~
	\sum_{n=1}^{\infty} \frac{1}{n^{s}} \cdot \frac{1}{\omega(n)+1} \cdot [ \sum_{p \in \mathbb{P}} \frac{1}{p^{s}} + \frac{1}{\omega(n)} \cdot \sum_{p | n} \frac{1}{p^{s}}   ]  = \zeta(s)-1
	\]

\end{restatable}

This result is about derivation operators over the ring $ (\mathbb{M},  \square  , \times )  $

\begin{restatable}[Character derivation theorem]{thm}{derivation}\label{thm:derivation}
\textnormal{Any derivation } $ \nabla $  \textnormal{ acting on the ring} $ (\mathbb{M},  \scalebox{0.6}{$\square$}  , \times ) $ \textnormal{is null over all dirichlet characters :}
	
		\[ \forall \chi ~~,~~ \nabla(\chi) = \numbermathbb{0} \]

\end{restatable}

The following corollaries are some direct application of the ring operations 

\begin{restatable}{prop}{orthproduct}\label{prop:orthproduct}

	\textnormal{for F and G completly multiplicatives with $ \delta_{1}(1) = 1  $ and  $ \forall n \in \mathbb{N}^{*} - \left\{ 1 \right\} :  \delta_{1}(n) = 0 $ :}
	\[ \forall F,G \in \mathbb{M_{c}} ~: ~ F \times G = \delta_{1} \Rightarrow D(F,s) \times D(G,s) = D( F \square G ,s) \]

\end{restatable}

\begin{restatable}{prop}{primecompfactor}\label{prop:primecompfactor}

	\textnormal{Let $ A \in \mathbb{P} $ and $ \bar{A} \in \mathbb{P} $ be complementary in $ \mathbb{P} $  , let F be an arithmetic completly multiplicative  function :}
	\[ \forall F \in \mathbb{M_{c}} ~:~ D(\mathds{1}_{A} \times F,s) \times D(\mathds{1}_{\bar{A}} \times F,s) = D(F,s) \]
	
\end{restatable}

\begin{rem}
	\textnormal{This formula reflet the euler product, let's take zêta as example :}
	
	\begin{align*}
	D(\mathds{1},s)
	&= \zeta(s) \\
	&= \prod_{p \in \mathbb{P}} \frac{1}{1-\frac{1}{p^{s}}} \\
	&=  \prod_{p \in A} \frac{1}{1-\frac{1}{p^{s}}} \times \prod_{p \in \bar{A}} \frac{1}{1-\frac{1}{p^{s}}} \\
	&= D(\mathds{1}_{A} \times \mathds{1},s) \times D(\mathds{1}_{\bar{A}} \times \mathds{1} ,s)\\
	\end{align*}
	
\end{rem}

\begin{restatable}{prop}{twotime}\label{prop:twotime}

	\textnormal{Identity in $ (\mathbb{M},  \square  , \times )  $ : }
	\[[F ~ \square ~ F ] (m) = [{2}^{\omega(.)} \times F ](m) \]

\end{restatable}

\begin{restatable}{prop}{sumchar}\label{prop:sumchar}

	\[ \forall n \in \mathbb{N}-\{0,1\} ,~ \forall a \in \mathbb{N}-\{0,1\}  :  \bigg[ \stackrel{\phi(n)} { \mathlarger{ \mathlarger{ \underset{k=1}{ ~  \mathlarger{ \mathlarger{ \square } }  ~ } }}  }   \mathlarger{ \mathlarger{ \chi_{k} } }  \bigg] (a)  = \phi(n)^{\omega(a)} \cdot \delta_{n | [\text{Id}_{e} - \mathds{1}](a) } \]

\end{restatable}

\begin{rem}

	\textnormal{For $ y \in \mathbb{R} $ : }
	
	\textnormal{Note the functions $ n \in \mathbb{N^{\star}} ~ : ~  n \rightarrow \cos(y\ln(n)) $  and $ n \in \mathbb{N^{\star}} ~ : ~  n \rightarrow \sin(y\ln(n)) $  respectively   $\textnormal{Cosa}_{y} \textnormal{ et } \textnormal{Sina}_{y} ~  $ because they play a similar role to the trigonometric functions inside the ring $ (\mathbb{M},  \square  , \times )  $ : }
	\[  \textnormal{Cosa}_{y}^{2}
 ~ \square ~  \textnormal{Sina}_{y}^{2} = \mathds{1}  \]
	
	\textnormal{This function will be used during the proof , it have a nice reformulation using the unitary ring operations :}
	\begin{align*}
		 Q_{\chi}(l) 
		 &= \Bigg[ \delta_{l=1}(l) + \sum_{\underset{mk=l, m>k }{m \wedge k = 1} } 2\Re\Big(\frac{\chi(m)}{{m}^{iy}}\overline{\frac{\chi(k)}{{k}^{iy}}}\Big) \Bigg]\\
		 &=[\textnormal{Cosa}_{y}\times\Re\chi ~ \square ~ \textnormal{Sina}_{y}\times\Im\chi] [l]
	\end{align*}
	
	Avec :
	\[ \forall n \in \mathbb{N}^{\star} ~ \forall p \in \mathbb{P} ~: ~ \\
	\Re\chi [p^{n}] =  \Re(\chi(p^{v_{p}(n)}))\]
	\[ \forall n \in \mathbb{N}^{\star} ~ \forall p \in \mathbb{P} ~: ~ \\
	\Im\chi [p^{n}] =  \Im(\chi(p^{v_{p}(n)}))\]
	This formulation could recall a scalar product between two complex values
\end{rem}

\section{Existence and unicity of the ring $ ( \mathds{M} , \square , \times ) $ }
In this chapter, our goal is to build a specific ring of functions and prove its uniqueness across a range of similar possible structures. The operation used in the final ring has been studied in numerous articles such as \cite{VAIDYANATHASWAMY} \cite{Eckford} \cite{SITARAMAIAH} .
We will begin with a few definitions, then we will progressively investigate what conditions the convolution used in the ring should meet.

\begin{dfn}
	We will start by introducing the weight functions $ W $  
	$$
	\begin{array}{l rcl}
	W  :   {\mathbb{N}^{\star}}^{2}  \longrightarrow  \mathbb{C} \\
	\end{array}
	$$
	
\end{dfn}

This function is used as a weight in the following convolution

\begin{dfn}
	The W function define the operation $\, \underset{w}{\scalebox{0.6}{$\square$}} \, $ , witch we will call the W-convolution :
	$$ \forall F,G \in \mathbb{M} , \forall m \in \mathbb{N}^{\star}  :   [F  \, \underset{w}{\scalebox{0.6}{$\square$}} \, G](m) = \sum_{ab=m} F(a)G(b)W(a,b) $$

\end{dfn} 

\begin{rem}
For notation purpuses , we will write only  $\, \scalebox{0.6}{$\square$} \, $ .
\end{rem}

we intent to construct a ring in this section, so let's define it's multiplication operation

\begin{dfn}
	We define the operation (the multiplication) $ \times $ as follow : 
	$$ \forall F,G \in \mathbb{M} , \forall m \in \mathbb{N}^{\star}  :   [F \times G](m) = F(m)G(m) $$ 
	
\end{dfn}
many intermadiary results will be established before prooving our main results , we will start by setting some tools that will show to be usefull throughout the ring construction
\begin{dfn}
	We define the multiplicative indicatrice functions  as follow :  
	$$ \forall S \subset \mathbb{P} \times \mathbb{N}^{\star} :   \mathds{1}_{S}(p^n) = \left\{
	\begin{array}{ll}
	1 & \mbox{si } (p,n) \in S\cup\{1,1\}  \\
	0 & \mbox{sinon.}
	\end{array}
	\right. $$ 
	we write also $ \mathds{1}_{s} $ for $ s \in \mathbb{N} $ , that we define as follow :  $$ \mathds{1}_{s}=\mathds{1}_{S} \iff s=\prod_{(p,n)\in S} p^{n}$$
\end{dfn}

\begin{lem}[commutativity lemma]\label{lem:commutativity}
	\textnormal{The operation $ \, \scalebox{0.6}{$\square$} \, $ is commutative if and only if the function $ W(.,.) $ is commutative : }
	$$ \forall F,G \in \mathbb{M} , \forall m \in \mathbb{N}^{\star}  :   F \, \scalebox{0.6}{$\square$} \, G(m) = G \, \scalebox{0.6}{$\square$} \, F(m)   \iff \forall a,b \in \mathbb{N}^{\star}  :   W(a,b) = W(b,a) $$
\end{lem}

\begin{proof}
	Let's start by the direct implication $\Longrightarrow$  : 
	\begin{align*}
	\mathds{1}_{a}  \, \scalebox{0.6}{$\square$} \,  \mathds{1}_{b}(ab) &= \sum_{nq=ab}\mathds{1}_{a}(n)\mathds{1}_{b}(q)W(n,q) \\
	&= \mathds{1}_{a}(a)\mathds{1}_{b}(b)W(a,b) \\
	&= W(a,b) 
	\end{align*}
	On the other direction : 
	\begin{align*}
	\mathds{1}_{b} \, \scalebox{0.6}{$\square$} \, \mathds{1}_{a}(ab)
	&= \sum_{nq=ab}\mathds{1}_{b}(n)\mathds{1}_{a}(q)W(n,q)\\
	&= \mathds{1}_{b}(b)\mathds{1}_{a}(a)W(b,a) \\
	&= W(b,a)
	\end{align*}
	So we have the first direction. 
	\newline
	let's establish now the reciproque $ \Longleftarrow $, we suppose :

	$$ \forall a,b \in \mathbb{N}^{\star}  :   W(a,b) = W(b,a) $$

	We have , $ \forall F,G \in $ $\mathbb{M} $ $ , \forall m \in $ $ \mathbb{N}^{\star} $ $  $  : 
	\begin{align*}
	F \, \scalebox{0.6}{$\square$} \, G(m)
	&= \sum_{nq=m}F(n)G(q)W(n,q)\\
	&= \sum_{\underset{n>q}{nq=m}}F(n)G(q)W(n,q)+\sum_{\underset{n<q}{nq=m}}F(n)G(q)W(n,q)+\sum_{\underset{n=q}{nq=m}}F(n)G(q)W(n,q) \\
	&= \sum_{\underset{q>n}{qn=m}}G(q)F(n)W(q,n)+\sum_{\underset{q<n}{qn=m}}G(q)F(n)W(q,n)+\sum_{\underset{q=n}{qn=m}}G(q)F(n)W(q,n) \\
	&= \sum_{nq=m}G(n)F(q)W(n,q)\\
	&= G \, \scalebox{0.6}{$\square$} \, F(m)
	\end{align*}
	So we have the reciproque . \\
	witch end the proof.
\end{proof}

\begin{lem}[stability lemma]\label{lem:stability}
	\textnormal{$ (\mathbb{M}, \, \scalebox{0.6}{$\square$} \,) $ is stable if and only if the function W is multiplicative over two variable in the following sens : } 
	\[ \forall F,G \in \mathbb{M}  :   F \, \scalebox{0.6}{$\square$} \, G \in \mathbb{M} \iff \forall a,b,c,d \in \mathbb{N}^{\star} ,\, ( ab , cd ) = 1 :  W(a,b)W(c,d)=W(ac,bd) \]
\end{lem}

\begin{proof}
	Let's start by the direct implication  $\Longrightarrow$  : 
	
	$ \forall a_{1},a_{2},b_{1},b_{2} \in \mathbb{N^{\star}} $ we set :   $  a_{1}a_{2}=a $ and $ b_{1}b_{2}=b $ .
	
	suppose $ (a ,  b) = 1 $.
	\begin{align*}
	\mathds{1}_{a_{1}b_{1}} \, \scalebox{0.6}{$\square$} \, \mathds{1}_{a_{2}b_{2}}(ab)
	&= \sum_{nq=ab}\mathds{1}_{a}(n)\mathds{1}_{b}(q)W(n,q)\\
	&= \mathds{1}_{a}(a_{1}a_{2})\mathds{1}_{b}(b_{1}b_{2})W(a_{1}a_{2},b_{1}b_{2}) \\
	&=W(a_{1}a_{2},b_{1}b_{2})
	\end{align*}
	as $ \mathds{1}_{a_{1}b_{1}} \, \scalebox{0.6}{$\square$} \, \mathds{1}_{a_{2}b_{2}} $ is a multiplicative function : 
	\begin{align*}
	\mathds{1}_{a_{1}b_{1}} \, \scalebox{0.6}{$\square$} \, \mathds{1}_{a_{2}b_{2}}(ab) 
	&= [\mathds{1}_{a_{1}b_{1}} \, \scalebox{0.6}{$\square$} \, \mathds{1}_{a_{2}b_{2}}](a)\times [\mathds{1}_{a_{1}b_{1}} \, \scalebox{0.6}{$\square$} \, \mathds{1}_{a_{2}b_{2}}](b) \\
	&= [\sum_{nq=a}\mathds{1}_{a_{1}b_{1}}(n)\mathds{1}_{a_{2}b_{2}}(q)W(n,q)]\times [\sum_{nq=b}\mathds{1}_{a_{1}b_{1}}(n)\mathds{1}_{a_{2}b_{2}}(q)W(n,q)] \\
	&= [\mathds{1}_{a_{1}b_{1}}(a_{1})\mathds{1}_{a_{2}b_{2}}(a_{2})W(a_{1},a_{2})]\times[\mathds{1}_{a_{1}b_{1}}(b_{1})\mathds{1}_{a_{2}b_{2}}(b_{2})W(b_{1},b_{2})] \\
	&= W(a_{1},a_{2})W(b_{1},b_{2})
	\end{align*}
	\textnormal{ So we have the first direction : }
	\[ W(a_{1}a_{2},b_{1}b_{2}) = W(a_{1},a_{2})W(b_{1},b_{2}) \]
	let's establish now the reciproc $ \Longleftarrow $  : 
	\\
	\textnormal{We suppose : } 
	\[ \forall a,b,c,d \in \mathbb{N}^{\star} ,\, ( ab , cd ) = 1  :   W(a,b)W(c,d)=W(ac,bd) \]
	\textnormal{We have : }
	\begin{align*}
	\forall F,G \in \mathbb{M} , \forall a,b \in \mathbb{N}^{\star} , ~ (a ,  b) =1 \textnormal{ : } [F \, \scalebox{0.6}{$\square$} \, G](ab)
	& = \sum_{nq=ab}F(n)G(q)W(n,q)\\
	& = \sum_{nq=a_{1}a_{2}b_{1}b_{2}}F(n)G(q)W(n,q)
	\end{align*}
	\textnormal{On the other direction : }
	\begin{align*}
	[F \, \scalebox{0.6}{$\square$} \, G](a)\times [F \, \scalebox{0.6}{$\square$} \, G](b)
	&= [\sum_{n_{1}n_{2}=a}F(n_{1})G(n_{2})W(n_{1},n_{2})]\times [\sum_{q_{1}q_{2}=b}F(q_{1})G(q_{2})W(q_{1},q_{2})]\\
	&= \sum_{n_{1}n_{2}=a}F(n_{1})G(n_{2})W(n_{1},n_{2})\sum_{q_{1}q_{2}=b}F(q_{1})G(q_{2})W(q_{1},q_{2})\\
	&= \sum_{n_{1}n_{2}=a}\sum_{q_{1}q_{2}=b}F(n_{1})G(n_{2})W(n_{1},n_{2})F(q_{1})G(q_{2})W(q_{1},q_{2})\\
	&= \sum_{n_{1}n_{2}=a}\sum_{q_{1}q_{2}=b}F(n_{1}q_{1})G(n_{2}q_{2})W(n_{1}q_{1},n_{2}q_{2})
	\end{align*}
	We have to prove :

	\[ \sum_{nq=a_{1}a_{2}b_{1}b_{2}}F(n)G(q)W(n,q) = \sum_{n_{1}n_{2}=a}\sum_{q_{1}q_{2}=b}F(n_{1}q_{1})G(n_{2}q_{2})W(n_{1}q_{1},n_{2}q_{2}) \]

	it is equivalent to prove, knowing that $ a_{1}a_{2}=a $ , $ b_{1}b_{2}=b $ and $ (a ,  b) = 1 $ that the variable change is bijectif. \\
	$
	 \textnormal{étant donné }   n_{1} , n_{2} , q_{1} , q_{2} \in \mathbb{N}^{\star}  \textnormal{ qui vérifie }  n_{1}n_{2}=a  \textnormal{ et }  q_{1}q_{2}=b  .  \\ 	$
	
	 let's set $ n= n_{1}q_{1} $ and  $ q=n_{2}q_{2} $ , we have the following equality :

	 $$ F(n_{1}q_{1}) G(n_{2}q_{2}) W(n_{1}q_{1},n_{2}q_{2}) = F(n)G(q)W(n,q) $$
	 
	where nq=ab . \\
	 invertly , knowing that $ n, q \in \mathbb{N}^{\star} $ witch verify :  nq=ab. \\
	we have the following equality : 
	\[ F(n)G(q)W(n,q)=F((n ,  a).(n ,  b))G((q ,  a).(q ,  b))W((n ,  a).(n ,  b),(q ,  a).(q ,  b)) \]
	where $ (n ,  a).(q ,  a)=a $ and $ (n ,  b).(q ,  b)=b $ . \\
	\textnormal{witch establish the equality , so we have the reciprocal direction. }
	
\end{proof}

\begin{lem}[identity lemma]\label{lem:identity}
	\textnormal{let's suppose the stability and commutativity of the structure $ (\mathbb{M},\, \scalebox{0.6}{$\square$} \,) $ , the identity element is the function $ \delta_{1} = \mathds{1}_{\varnothing} $ : }
	$$ \exists E \in \mathbb{M} \, \forall F \in \mathbb{M}   :   F \, \scalebox{0.6}{$\square$} \, E = E  \, \scalebox{0.6}{$\square$} \, F = F \iff \forall (n,p) \in \mathbb{N}\times \mathbb{P} ,\, W(1,p^{n})=1 , E = \delta_{1}$$
\end{lem}

\begin{proof}
	$ $\newline
	using commutativity , we have : 
	$$ \exists E \in \mathbb{M} \, \forall F \in \mathbb{M}   :   F \, \scalebox{0.6}{$\square$} \, E = E  \, \scalebox{0.6}{$\square$} \, F = F \iff \exists E \in \mathbb{M} \, \forall F \in \mathbb{M}   :   F \, \scalebox{0.6}{$\square$} \, E = F $$
	using stability , we could  restrict the proof to : 
	$$ \exists E \in \mathbb{M} \, \forall F \in \mathbb{M}   :   F \, \scalebox{0.6}{$\square$} \, E = F \iff \forall F \in \mathbb{M} , \forall (n,p) \in \mathbb{N}^{\star}\times \mathbb{P} ,\,   :   [F \, \scalebox{0.6}{$\square$} \, E](p^{n}) = F(p^{n}) $$
	we then compute the following  W-convolution  : 
	$ \forall (n,p) \in \mathbb{N}\times \mathbb{P}  :  $ 
	\begin{align*}
	[F \, \scalebox{0.6}{$\square$} \, E](p^{n})
	&= \sum_{nq=p^{n}}F(n)E(q)W(n,q)\\
	&= \sum_{l=0}^{l=n}F(p^{l})E(p^{n-l})W(p^{l},p^{n-l}) 
	\end{align*}
	let's fix $ p \in \mathbb{P} $ and proof the direct proposition by (strong) reccurence : 
	$$ \forall F \in \mathbb{M} , \forall n \in \mathbb{N} ,\,   :   [F \, \scalebox{0.6}{$\square$} \, E](p^{n}) = F(p^{n}) \Longrightarrow  \forall n \in \mathbb{N}^{\star} ,\, W(1,p^{n})=1 ~ \textnormal{et} ~ E = \delta_{1} $$
	initial case :  $ n=0 $
	\begin{align*}
	[F \, \scalebox{0.6}{$\square$} \, E](1)
	&= W(1,1) 
	\end{align*}
	\textnormal{so we get} 
	$$ W(1,1) = 1 ~ \textnormal{ et } ~ E(1) = 1 $$
	\textnormal{ reccurence hypothesis : }
	$$ \forall F \in \mathbb{M} , \exists n_{0} \in \mathbb{N} ,\,   :   [F \, \scalebox{0.6}{$\square$} \, E](p^{n_{0}}) = F(p^{n_{0}}) \Longrightarrow  \forall n \in \mathbb{N}^{\star} ,\, W(1,p^{n_{0}})=1 ~ \textnormal{et} ~ E = \delta_{1} $$
	n+1 case : 
	
	\begin{align*}
	[F \, \scalebox{0.6}{$\square$} \, E](p^{n+1})
	&= \sum_{l=0}^{l=n+1}F(p^{l})E(p^{n+1-l})W(p^{l},p^{n+1-l}) \\
	&= F(p^{n+1})E(1)W(p^{n+1},1)+F(1)E(p^{n+1})W(1,p^{n+1}) \\
	&= W(p^{n+1},1)(F(p^{n+1})+E(p^{n+1}))
	\end{align*}
	\textnormal{so  : }
	\[W(p^{n+1},1)(F(p^{n+1})+E(p^{n+1}))-F(p^{n+1})  
	= F(p^{n+1})(W(p^{n+1},1)-1)+E(p^{n+1})W(p^{n+1},1)\]
	as nether the identity element  $ E $ , nor the weight function $ W $ does not depend on  $ F $ , this expression could be seen as a polynôme of  $ F(p^{n+1}) $ .\\
	\textnormal{Let's set  $ A=W(p^{n+1},1)-1 ~ B=E(p^{n+1})W(p^{n+1},1) $ , we the have  : }
	\[ A.F(p^{n+1})+B = 0 \iff A = 0 \, \text{et} \, B=0 \]
	so :
	\[ W(p^{n+1},1) = 1 ~ \textnormal{and} ~ E(p^{n+1}) = 0 \]
	so we have the first proposition. \\
	\textnormal{inversly : }
	\begin{align*}
	[F \, \scalebox{0.6}{$\square$} \, \delta_{1}](p^{n})
	&= \sum_{l=0}^{l=n}F(p^{l})E(p^{n-l})W(p^{l},p^{n-l}) \\
	&= F(p^{n})E(1)W(p^{n+1},1) \\
	&= F(p^{n})
	\end{align*}
	witch end our proof
\end{proof}

\begin{lem}[associativity lemma]\label{lem:associativity}
	\textnormal{ Hypothesis : } 
	\textnormal{ $ (\mathbb{M}, \scalebox{0.6}{$\square$} ) $ is stable and commutative . } \\
	\textnormal{ the structure is associative $ (\mathbb{M}, \scalebox{0.6}{$\square$} ) $ is equivalent to : }
	\[  \forall F,G,H \in \mathbb{M}   :   [[F \, \scalebox{0.6}{$\square$} \, G] \, \scalebox{0.6}{$\square$} \, H] = [F \, \scalebox{0.6}{$\square$} \, [G \, \scalebox{0.6}{$\square$} \, H]] \iff \forall a,b,c \in \mathbb{N}^{\star} ,\, W(a,b)W(ab,c)=W(b,c)W(bc,a) \]
\end{lem}

\begin{proof}
	\textnormal{ the direct proposition $ \Longrightarrow $  could be prooven using the indicatrice functions : }
	\begin{align*}
	\forall a,b,c \in \mathbb{N}^{\star} \,\mathds{1}_{a} \, \scalebox{0.6}{$\square$} \, [\mathds{1}_{b}\, \scalebox{0.6}{$\square$} \,\mathds{1}_{c}](abc) 
	&= \sum_{efg=abc}\mathds{1}_{a}(f)\mathds{1}_{b}(g)\mathds{1}_{c}(e)W(g,e)W(ge,f) \\
	&= W(b,c)W(bc,a)
	\end{align*}
	and 
	\begin{align*}
	\forall a,b,c \in \mathbb{N}^{\star} \, [\mathds{1}_{a} \, \scalebox{0.6}{$\square$} \, \mathds{1}_{b}]\, \scalebox{0.6}{$\square$} \,\mathds{1}_{c}(abc)
	&= \sum_{efg=abc}\mathds{1}_{a}(f)\mathds{1}_{b}(g)\mathds{1}_{c}(e)W(f,g)W(fg,e) \\
	&= W(a,b)W(ab,c)
	\end{align*}
	so we have  : 
	\[ \forall a,b,c \in \mathbb{N}^{\star}  :   W(a,b)W(ab,c) = W(b,c)W(bc,a) \]
	the reciproc $ \Longleftarrow $  : 
	
	$ \forall F,G,H \in \mathbb{M} , \forall (n,p) \in \mathbb{N^{\star}}\times \mathbb{P}  :  $
	\begin{align*}
	[F \, \scalebox{0.6}{$\square$} \, G]\, \scalebox{0.6}{$\square$} \,H(p^{n})
	&= \sum_{ab=p^{n}}[F \, \scalebox{0.6}{$\square$} \, G](a)H(b)W(a,b) \\
	&= \sum_{ab=p^{n}}\sum_{cd=a}F(c)G(d)W(c,d)H(b)W(cd,b) \\
	&= \sum_{bcd=p^{n}}F(c)G(d)H(b)W(c,d)W(cd,b) 
	\end{align*}
	and  : 
	\begin{align*}
	F \, \scalebox{0.6}{$\square$} \, [G\, \scalebox{0.6}{$\square$} \,H](p^{n})
	&= \sum_{ba=p^{n}}F(b)[G \, \scalebox{0.6}{$\square$} \, H](a)W(b,a) \\
	&= \sum_{ba=p^{n}}F(b)\sum_{cd=a}G(c)H(d)W(c,d)W(b,a) \\
	&= \sum_{bcd=p^{n}}F(b)G(c)H(d)W(c,d)W(b,cd) \\
	&= \sum_{bcd=p^{n}}F(c)G(d)H(b)W(d,b)W(c,db) \\ 
	&= \sum_{bcd=p^{n}}F(c)G(d)H(b)W(d,b)W(db,c) 
	\end{align*}
	but we have  :  
	\[ \forall a,b,c \in \mathbb{N}^{\star}  :   W(a,b)W(ab,c) = W(b,c)W(bc,a) \]
	so : 
	$$\sum_{bcd=p^{n}}F(c)G(d)H(b)W(c,d)W(cd,b)=\sum_{bcd=p^{n}}F(c)G(d)H(b)W(d,b)W(db,c) $$
	wicht gives us the direct propositiont .
	
\end{proof} 

\begin{lem}[multiplication properties]\label{lem:multiplication}
	\textnormal{ let's note : } $\mathds{1}  : = \mathds{1}_{\mathbb{P}\times\mathbb{N}}   $.
	
	\textnormal{$ (\mathbb{M},\, \times \,) $ is stable , commutatif , associative and admit $ \mathds{1} $ as the identity element . }
\end{lem}

\begin{proof}
	Stability : 
	$$ \forall F,G \in \mathbb{M}, \forall a,b \in \mathbb{N^{\star}}^{2}  :  [F \times G](ab)=F(ab)\times G(ab)=F(a)G(a)F(b)G(b)=[F\times G](a)\times[F\times G](b) $$
	so  : 
	$$ \forall F,G \in \mathbb{M}  :  F\times G \in \mathbb{M} $$
	Commutativity :  
	$$ \forall F,G \in \mathbb{M}, \forall m \in \mathbb{N^{\star}} :  [F \times G](m)=F(m)\times G(m)=G(m)\times F(m)=[G \times F](m) $$
	so : 
	$$ \forall F,G \in \mathbb{M}  :  F \times G = G \times F  $$
	Associativity : 
\begin{align*} \forall F,G,H \in \mathbb{M}, \forall m \in \mathbb{N^{\star}}  :  [[F \times G] \times H ](m) & =[F\times G](m) \times H(m) \\ 
& =F(m)G(m)H(m) \\ 
& =[F\times [G\times H ]](m) 
\end{align*}
	so : 
	$$ \forall F,G,H \in \mathbb{M}  :  [[F \times G] \times H = [F\times [G\times H ]]  $$
	$ \mathds{1} $ is the identity element : 
	$$ \forall F \in \mathbb{M} , \forall (n,p) \in \mathbb{N}\times \mathbb{P}  :  [F\times \mathds{1}](p^{n}) = F(p^{n})\times\mathds{1}(p^{n}) = F(p^{n}) $$
	so ,  by unicity of the identity element , $ \mathds{1} $ is the identity element of  $ (\mathbb{M},\, \times \,) $ . \\
	witch proves the theorem.
	
\end{proof}

\begin{lem}
	\textnormal{Hypothesis : }
	\[ \forall a,b,c,d \in \mathbb{N}^{\star} ,\, ( ab , cd ) = 1  :   W(a,b)W(c,d)=W(ac,bd) \]
	\textnormal{witch could be generelised as : }
	\begin{align*}
	\forall n \in \mathbb{N} \, , \, \forall \, {a}_{0},{a}_{1}, ... , {a}_{n} \in \mathbb{N}^{\star} , \, &\forall \, {b}_{0},{b}_{1}, ... , {b}_{n} \in \mathbb{N}^{\star} \, | \, \forall i,j \in \llbracket 0,n\rrbracket \, , \,( {a}_{i}{b}_{i}  ,  {a}_{j}{b}_{j} ) = 1 \\
	& \prod_{i=0}^{i=n}W({a}_{i},{b}_{i}) = W(\prod_{i=0}^{i=n}{a}_{i},\prod_{i=0}^{i=n}{b}_{i})
	\end{align*}
\end{lem}
\begin{proof}
	we will procced by reccurence : \\
	\textnormal{Cas $n=0$ : }
	\[ \prod_{i=0}^{i=0}W({a}_{i},{b}_{i}) = W(\prod_{i=0}^{i=0}{a}_{i},\prod_{i=0}^{i=0}{b}_{i}) 
	 \iff 1 = W(1,1) \]
	the hypothesis of reccurence : 
	\begin{align*}
	\exists {n}_{0} \in \mathbb{N} \, , \, \forall \, {a}_{0},{a}_{1}, ... , {a}_{{n}_{0}} \in \mathbb{N}^{\star} , \, &\forall \, {b}_{0},{b}_{1}, ... , {b}_{{n}_{0}} \in \mathbb{N}^{\star} \, | \, \forall i,j \in \llbracket 0,{n}_{0}\rrbracket \, , \,( {a}_{i}{b}_{i}  ,  {a}_{j}{b}_{j} ) = 1 \\
	& \prod_{i=0}^{i={n}_{0}}W({a}_{i},{b}_{i}) = W(\prod_{i=0}^{i={n}_{0}}{a}_{i},\prod_{i=0}^{i={n}_{0}}{b}_{i})
	\end{align*}
	The case $n+1$  : 
	\begin{align*}
	\prod_{i=0}^{i=n+1}W({a}_{i},{b}_{i}) &= [\prod_{i=0}^{i=n}W({a}_{i},{b}_{i})]W({a}_{n+1},{b}_{n+1}) \\
	&=W(\prod_{i=0}^{i=n}{a}_{i},\prod_{i=0}^{i=n}{b}_{i})W({a}_{n+1},{b}_{n+1}) \tag*{[reccurence hypothesis]}\\
	&=W(\prod_{i=0}^{i=n+1}{a}_{i},\prod_{i=0}^{i=n+1}{b}_{i})\tag*{[$({a}_{n+1}{b}_{n+1} , \prod_{i=0}^{i=n}{a}_{i}\prod_{i=0}^{i=n}{b}_{i})=1$]}
	\end{align*}
	\textnormal{So we have the results we wanted. }
\end{proof}

\begin{lem}
	\textnormal{Hypothesis  :  $ (\mathbb{M}, \, \scalebox{0.6}{$\square$} \,) $ is stable , commutative ,  and have an identity element. }\\
	\textnormal{The function W could be written as   : } 
	\begin{align*}
	\forall n,q \in \mathbb{N}^{\star}  :  W(n,q) &=\prod_{\stackrel{p | \frac{nq}{(n, q)^{2}}}{p \in \mathbb{P}}} W({p}^{v_{p}(\frac{nq}{(n, q)^{2}})},1)\prod_{\stackrel{p | ( n , q )}{p \in \mathbb{P}}}W({p}^{v_{p}(( n , q ))},{p}^{v_{p}(( n , q ))})\\
	&=\prod_{\stackrel{p | ( n , q )}{p \in \mathbb{P}}}W({p}^{v_{p}(( n , q ))},{p}^{v_{p}(( n , q ))})
	\end{align*}
\end{lem}

\begin{rem}
	This means that in this case, the function W is entirely determined by the image of equal prime power pairs.
\end{rem}

\begin{proof}
	\textnormal{the W function have , by hypothesis , the following property : }
	\[ \forall a,b,c,d \in \mathbb{N}^{\star} ,\, ( ab , cd ) = 1  :   W(a,b)W(c,d)=W(ac,bd) \]
	We have $ (\frac{nq}{( n , q )^{2}} , (( n , q ))^{2}) = 1 $ so  : 
	\begin{align*}
	\forall n,q \in \mathbb{N}^{\star}  :  W(n,q)
	&=W(\frac{n}{( n , q )},\frac{q}{( n , q )})W(( n , q ),( n , q ))
	\end{align*}
	let's compute the first part : 
	\begin{align*}
	\forall n,q \in \mathbb{N}^{\star}  :  W(\frac{n}{( n , q )},\frac{q}{( n , q )}) 
	&=W(\frac{n}{( n , q )},1)W(1,\frac{q}{( n , q )})\\
	&=W(\frac{nq}{{(( n , q ))}^{2}},1)\tag*{$ (\frac{n}{( n , q )}\times 1 , \frac{q}{( n , q )}\times 1)=1 $}\\
	&=\prod_{\underset{p \in \mathbb{P}}{p | \frac{nq}{{(( n , q ))}^{2}}}}W({p}^{v_{p}(\frac{nq}{{(( n , q ))}^{2}})},1)
	\end{align*}
	let's compute the second part : 
	\begin{align*}
	\forall n,q \in \mathbb{N}^{\star}  :  W(( n , q ),( n , q ))
	&= \prod_{\underset{p \in \mathbb{P}}{p | ( n , q )}}W({p}^{v_{p}(( n , q ))},{p}^{v_{p}(( n , q ))})
	\end{align*}
	so  : 
	\begin{align*}
	\forall n,q \in \mathbb{N}^{\star}  :  W(n,q)
	&=\prod_{\underset{p \in \mathbb{P}}{p | \frac{nq}{{( n , q )}^{2}}}}W({p}^{v_{p}(\frac{nq}{{( n , q )}^{2}})},1)\prod_{\underset{p \in \mathbb{P}}{p | ( n , q )}}W({p}^{v_{p}(( n , q ))},{p}^{v_{p}(( n , q ))})
	\end{align*}
	by hypothesis ( the existence of the identity element ) , we have that  : 
	\begin{align*}
	\forall n,q \in \mathbb{N}^{\star}  :  W(n,q)
	&=\prod_{\underset{p \in \mathbb{P}}{p | ( n , q )}}W({p}^{v_{p}(( n , q ))},{p}^{v_{p}(( n , q ))})
	\end{align*}
\end{proof}

\begin{lem}[distributivity lemma]\label{lem:distributivity}
	\textnormal{ Hypothesis : Commutativity , stability , associativity of the two opérations in the structure $(\mathbb{M},  \scalebox{0.6}{$\square$}  , \times ) $ . } \\
	\textnormal{In the structure $(\mathbb{M},  \scalebox{0.6}{$\square$}  , \times ) $ , The operation  \scalebox{0.6}{$\square$}  is distributive by  $ \times $  if and only if the W function , the weight of the convolution , follow some conditions : }
	\[ \forall F,G,H \in \mathbb{M}  :  [F \, \scalebox{0.6}{$\square$} \, G] \times H = [F \times H ] \, \scalebox{0.6}{$\square$} \, [G \times H] \iff  \forall a,b \in \mathbb{N}^{\star}  :  W(a,b) =\begin{cases}
	\begin{array}{ll}
	1 \,\textrm{ \textnormal{if} } \, (a ,  b) = 1 \\
	0 \,\textrm{ \textnormal{if not}}
	\end{array}
	\end{cases}  \]
	
\end{lem}

\begin{proof}
	Let's start by the firts direction $ \Longrightarrow $ : 
	\begin{align*}
	\forall \, l,f  \in \mathbb{N} \, l+f=n , \forall (n,p) \in \mathbb{N^{\star}}\times \mathbb{P}  :  \\ [\mathds{1}_{{p}^{l}} \, \scalebox{0.6}{$\square$} \, \mathds{1}_{{p}^{f}}] \times \mathds{1}_{{p}^{n}} ({p}^{n}) 
	&= [\mathds{1}_{{p}^{l}} \, \scalebox{0.6}{$\square$} \, \mathds{1}_{{p}^{f}}] ({p}^{n}) \times \mathds{1}_{{p}^{n}} ({p}^{n}) \\ 
	&=\sum_{ab=p^{n}}\mathds{1}_{{p}^{l}}(a)\mathds{1}_{{p}^{f}}(b)W(a,b) \\
	&=W({p}^{l},{p}^{f})
	\end{align*}
	\begin{align*}
	\forall \, l,f  \in \mathbb{N} \, l+f=n , \forall (n,p) \in \mathbb{N^{\star}}\times \mathbb{P}  :  \\ [\mathds{1}_{{p}^{l}} \times \mathds{1}_{{p}^{n}} ] \, \scalebox{0.6}{$\square$} \, [\mathds{1}_{{p}^{f}} \times \mathds{1}_{{p}^{n}} ] ({p}^{n}) 
	&=\sum_{ab=p^{n}}[\mathds{1}_{{p}^{l}} \times \mathds{1}_{{p}^{n}} ](a) [\mathds{1}_{{p}^{f}} \times \mathds{1}_{{p}^{n}} ] (b) W(a,b) \\ 
	&=\sum_{ab=p^{n}} \mathds{1}_{{p}^{l}}(a) \mathds{1}_{{p}^{n}}(a) \mathds{1}_{{p}^{f}}(b) \mathds{1}_{{p}^{n}}(b) W(a,b) \\
	&=\mathds{1}_{{p}^{l}}({p}^{n}) \mathds{1}_{{p}^{n}}({p}^{n}) \mathds{1}_{{p}^{f}}(1) \mathds{1}_{{p}^{n}}(1) W({p}^{n},1) \\
	&+ \mathds{1}_{{p}^{l}}(1) \mathds{1}_{{p}^{n}}(1) \mathds{1}_{{p}^{f}}({p}^{n}) \mathds{1}_{{p}^{n}}({p}^{n}) W(1,{p}^{n})\\
	&=[\mathds{1}_{{p}^{l}}({p}^{n}) +  \mathds{1}_{{p}^{f}}({p}^{n})] W({p}^{n},1)
	\end{align*} 
	case $l.f=0$ :  \\
	we then have $  l=n ~ \textrm{or} ~ f=n$ : 
	\[
	[\mathds{1}_{{p}^{l}}({p}^{n}) +  \mathds{1}_{{p}^{f}}({p}^{n})] W({p}^{n},1)
	=[\mathds{1}_{{p}^{n}}({p}^{n}) +  \mathds{1}_{{p}^{0}}({p}^{n})] W({p}^{n},1) 
	=1
	\] 
	Case $l.f \neq 0$ : \\
	we then have $ \, l \neq 0 \, \textrm{,} \, f \neq 0 \, \textrm{,} \,  l \neq n \, \textrm{,} \, f \neq n \, $ and :
	\begin{align*}
	[\mathds{1}_{{p}^{l}}({p}^{n}) +  \mathds{1}_{{p}^{f}}({p}^{n})] W({p}^{n},1)
	&=[0+0] W({p}^{n},1) \\
	&=0
	\end{align*}
	So to sum up :
	\[
	W({p}^{l},{p}^{f}) =\begin{cases}
	\begin{array}{ll}
	1 \,\textrm{ if } \, l.f=0 \\
	0 \,\textrm{ if not}
	\end{array}
	\end{cases}
	\]
	In other words (equivalently) :
	\[
	W({p}^{l},{p}^{f}) =\begin{cases}
	\begin{array}{ll}
	1 \,\textrm{ if } \, ({p}^{l} , {p}^{f}) = 1 \\
	0 \,\textrm{ if not}
	\end{array}
	\end{cases}
	\]
	Following the precedent result, this determine entirely the function , but the function  : 
	\[
	\forall a,b \in \mathbb{N}^{\star}  :   W(a,b) =
	\begin{cases}
	\begin{array}{ll}
	1 \,\textrm{ if } \, (a ,  b) = 1 \\
	0 \,\textrm{ if not}
	\end{array}
	\end{cases}
	\]
	verify those conditions , so the direct sens.\\
	let's now check the inderect sens $ \Longleftarrow $  : 
	\begin{align*}
	\forall F,G,H \in \mathbb{M} \, \forall (n,p) \in \mathbb{N^{\star}}\times \mathbb{P}  : \\ [F \, \scalebox{0.6}{$\square$} \, G] \times H ({p}^{n}) &= [\sum_{\underset{(a ,  b) = 1}{ab={p}^{n}}} F(a)G(b)W(a,b)]H({p}^{n}) \\
	&=[F(1)G({p}^{n})W(1,{p}^{n})+F({p}^{n})G(1)W({p}^{n},1)]H({p}^{n})\\
	&=G({p}^{n})H({p}^{n})+F({p}^{n})H({p}^{n})
	\end{align*} 
	\begin{align*}
	\forall F,G,H \in \mathbb{M} \, \forall (n,p) \in \mathbb{N^{\star}}\times \mathbb{P}  : \\ [F \times H ] \, \scalebox{0.6}{$\square$} \, [G \times H] ({p}^{n}) &= \sum_{\underset{(a ,  b) = 1}{ab={p}^{n}}} F(a)H(a)G(b)H(b)W(a,b) \\
	&=F(1)H(1)G({p}^{n})H({p}^{n})W(1,{p}^{n})+F({p}^{n})H({p}^{n})G(1)H(1)W({p}^{n},1)\\
	&=G({p}^{n})H({p}^{n})+F({p}^{n})H({p}^{n})
	\end{align*} 
	So we have the result we want to establish
\end{proof}

\begin{lem}[inverse lemma]\label{lem:inverse}
	\textnormal{Hypothesis :
	commutativity , stability , existance of the neutral element in the structure $(\mathbb{M},  \scalebox{0.6}{$\square$}  , \times ) $ and the distributivity of  $\scalebox{0.6}{$\square$}$ over $\times$.} \\
	\textnormal{In the structure $ (\mathbb{M},  \scalebox{0.6}{$\square$}  , \times ) $  , for all element of  $\mathbb{M}$ , there exist a unique inverse according to the operation $\scalebox{0.6}{$\square$}$ defined as : }
	\[ \forall F \in \mathbb{M}  :  [F \, \scalebox{0.6}{$\square$} \, I_{F}]  = \delta_{1} \iff  \forall  n,p \in \mathbb{N}^{\star} \times \mathbb{P} \,  :  \, I_{F}({p}^{n})=-F({p}^{n}) \]
\end{lem}

\begin{proof}
	let's compute  : 
	\begin{align*}
	\forall F \in \mathbb{M} \forall  n,p \in \mathbb{N}^{\star} \times \mathbb{P}  :  [F \, \scalebox{0.6}{$\square$} \, I_{F}]({p}^{n}) &= \sum_{\underset{(a ,  b) = 1}{ab = {p}^{n}}} F(a) I_{F}(b)W(a,b) \\
	&= F(1)I_{F}({p}^{n}) + F({p}^{n})I_{F}(1) \\
	&= I_{F}({p}^{n})+F({p}^{n}) \\
	&=\delta_{1}({p}^{n}) \\
	&=0
	\end{align*}
	So
	\[ \forall F \in \mathbb{M} \forall  n,p \in \mathbb{N}^{\star} \times \mathbb{P}  :  I_{F}({p}^{n}) = -F({p}^{n})\]
	In the case n=0  : 
	\begin{align*}
	\forall F \in \mathbb{M}   :  [F \, \scalebox{0.6}{$\square$} \, I_{F}](1) &= \sum_{\underset{(a ,  b) = 1}{ab = 1}} F(a) I_{F}(b)W(a,b) \\
	&= F(1)I_{F}(1) \\
	&=\delta_{1}(1) \\
	&=1
	\end{align*}
	So
	$$ I_{F}(1)=1  $$
	let's sum up :
	\begin{align*}
	\forall F \in \mathbb{M} \forall  n,p \in \mathbb{N} \times \mathbb{P}  :  I_{F}({p}^{n}) =  
	\begin{cases}
	\begin{array}{ll}
	1 \,\textrm{ if } \, n = 0 \\
	-F({p}^{n}) \,\textrm{ if not }
	\end{array}
	\end{cases}
	\end{align*}
	witch represent the same function here :
	\begin{align*}
	\forall F \in \mathbb{M} \forall  n \in \mathbb{N}^{\star}  :  I_{F}(n) = F^{-1}(n) = (-1)^{\omega(n)} F(n)
	\end{align*}
\end{proof}

In this section , we present the main results that we obtened and concluded using the precedent lemmas  : 

\unicity*

\begin{proof}
	by combining the precedent lemmas
\begin{align*}		
	\begin{split}
	& (\mathbb{M},  \scalebox{0.6}{$\square$}  , \times ) ~ \, \textnormal{is a commutative ring}
	  \\ & \Rightarrow   
		\begin{cases}
			\begin{array}{l rcl}
				\forall a,b \in \mathbb{N}^{\star}  :   W(a,b) = W(b,a) $ ( \nameref{lem:commutativity} :  \ref{lem:commutativity} ) $ \\ 
				\forall a,b,c,d \in \mathbb{N}^{\star} ,\, ( ab , cd ) = 1  :   W(a,b)W(c,d)=W(ac,bd) $ ( \nameref{lem:stability} :  \ref{lem:stability} ) $ \\
				\forall (n,p) \in \mathbb{N}\times \mathbb{P} ,\, W(1,p^{n})=1 $ ( \nameref{lem:identity} :  \ref{lem:identity} ) $ \\
				\forall a,b,c \in \mathbb{N}^{\star} ,\, W(a,b)W(ab,c)=W(b,c)W(bc,a) $ ( \nameref{lem:associativity} :  \ref{lem:associativity} ) $ \\
				\begin{array}{l rcl}
					\forall a,b \in \mathbb{N}^{\star}  :  W(a,b) =
					\begin{cases}
						\begin{array}{ll}
							1 \,\textnormal{ if } \, (a ,  b) = 1 \\
							0 \,\textnormal{ if not}
						\end{array}
					\end{cases} $ ( \nameref{lem:distributivity} :  \ref{lem:distributivity} ) $
				\end{array}
			\end{array}
		\end{cases}  \\
	& \Rightarrow
		\begin{cases}
			\begin{array}{l rcl}
			\forall a,b \in \mathbb{N}^{\star}  :  W(a,b) =
				\begin{cases}
					\begin{array}{ll}
						1 \,\textnormal{ if } \, (a ,  b) = 1 \\
						0 \,\textnormal{ if not}
					\end{array}
				\end{cases}
			\end{array} 
		\end{cases}
	\end{split}
\end{align*}
invertly (the checking follows by testing each of the properties) :
\begin{align*}		
\begin{split}
\begin{array}{l rcl}
\forall a,b \in \mathbb{N}^{\star}  :  W(a,b) =
\begin{cases}
\begin{array}{ll}
1 \,\textnormal{ if } \, (a ,  b) = 1 \\
0 \,\textnormal{ if not}
\end{array}
\end{cases}
\end{array} 
& \Rightarrow  
(\mathbb{M},  \scalebox{0.6}{$\square$}  , \times ) ~ \, \textnormal{is a commutative ring}\\
\end{split}
\end{align*}
\end{proof}

\section{Ring's propositions proofs}
We will prove in this section many theorems previously stated, we will also present propositions that are closely related to them
\refactoristion*

\begin{proof}
\begin{align*}
\forall F,G \in \mathbb{M}_{c} : D(F \star G,s)
	& = D(F,s) \times D(G,s)   \\
	& = \Bigg[ \sum_{n=1}^{+\infty} \frac{F(n)}{n^{s}} \Bigg] \times \Bigg[ \sum_{q=1}^{+\infty} \frac{G(q)}{q^{s}} \Bigg] \\
	& =  \sum_{n=1}^{\infty} \sum_{q=1}^{\infty} \frac{F(n)}{n^{s}}
	\frac{G(q)}{q^{s}} \\
	& = \sum_{\underset{m \wedge k =1}{m=1 k=1}}  \sum_{{r=1}}^{+\infty} 
	\frac{F(mr)}{(mr)^{s}}
	\frac{G(kr)}{(kr)^{s}} \\
	& = \sum_{{r=1}}^{+\infty} \frac{F(r)G(r)}{(r)^{2s}} \sum_{m \wedge k =1}  \Bigg[ \frac{F(m)}{(m)^{s}} 
	\frac{G(k)}{(k)^{s}} \Bigg] \\
	& = \sum_{{r=1}}^{+\infty} \frac{F(r)G(r)}{(r)^{2s}} \sum_{l=1}^{\infty}\sum_{\underset{mk=l}{m \wedge k =1}}  \Bigg[ \frac{F(m)}{(m)^{s}} 
	\frac{G(k)}{(k)^{s}} \Bigg] \\
	& = \Bigg[ \sum_{r=1}^{\infty} \frac{F(r)G(r)}{r^{2s}} \Bigg] \Bigg[ \sum_{l=1}^{\infty}\sum_{\underset{mk=l}{m \wedge k =1}}  \Bigg[ \frac{F(m)}{(m)^{s}} 
	\frac{G(k)}{(k)^{s}} \Bigg] \\ 
	& = \Bigg[ \sum_{r=1}^{\infty} \frac{F(r)G(r)}{r^{2s}} \Bigg] \Bigg[ \sum_{l=1}^{\infty}\frac{ \sum_{\underset{mk=l}{m \wedge k =1}}  F(m) G(k) }{l^{s}} \Bigg] \\ 
	& = \Bigg[ \sum_{r=1}^{\infty} \frac{F(r)G(r)}{r^{2s}} \Bigg] \Bigg[ \sum_{l=1}^{\infty}\frac{ F \square G (l) }{l^{s}} \Bigg] \\ 
	& = D(F \times G,2s) \times D( F \square G ,s)
\end{align*}
\end{proof}

The alternative version have a similar prove :
\realimsplit

\begin{proof}
	\begin{align*}
	F,G \in \mathbb{M}_{c} ~ \forall y \in \mathbb{C} : \\ D(F,s) \times D(G,\overline{s}) 
	& = \Bigg[ \sum_{n=1}^{+\infty} \frac{F(n)}{n^{s}} \Bigg] \times \Bigg[ \sum_{q=1}^{+\infty} \frac{G(q)}{q^{\overline{s}}} \Bigg] \\
	& =  \sum_{n=1}^{\infty} \sum_{q=1}^{\infty} \frac{F(n)}{n^{s}}
	\frac{G(q)}{q^{\overline{s}}} \\
	& = \sum_{\underset{m \wedge k =1}{m=1 k=1}}  \sum_{{p=1}}^{+\infty} 
	\frac{F(mp)}{(mp)^{s}}
	\frac{G(kp)}{(kp)^{\overline{s}}} \\
	& = \sum_{\underset{m \wedge k =1}{m=1 k=1}}  \sum_{{p=1}}^{+\infty} 
	\frac{F(mp)}{(mp)^{x+iy}}
	\frac{G(kp)}{(kp)^{x-iy}} \\
	& = \sum_{{p=1}}^{+\infty} \frac{F(p)G(p)}{(p)^{2x}} \sum_{m \wedge k =1}  \Bigg[ \frac{F(m)}{(m)^{x+iy}} 
	\frac{G(k)}{(k)^{x-iy}} \Bigg] \\
	& = \sum_{{p=1}}^{+\infty} \frac{F(p)G(p)}{(p)^{2x}} \sum_{l=1}^{\infty}\sum_{\underset{mk=l}{m \wedge k =1}}  \Bigg[ \frac{F(m)}{(m)^{x+iy}} 
	\frac{G(k)}{(k)^{x-iy}} \Bigg] \\
	& = \Bigg[ \sum_{p=1}^{\infty} \frac{F(p)G(p)}{p^{2x}} \Bigg] \Bigg[ \sum_{l=1}^{\infty}\sum_{\underset{mk=l}{m \wedge k =1}}  \Bigg[ \frac{F(m)}{(m)^{x+iy}} 
	\frac{G(k)}{(k)^{x-iy}} \Bigg] \\ 
	& = \Bigg[ \sum_{p=1}^{\infty} \frac{F(p)G(p)}{p^{2x}} \Bigg] \Bigg[ \sum_{l=1}^{\infty}\frac{ \sum_{\underset{mk=l}{m \wedge k =1}}  \frac{F(m)}{(m)^{iy}} 
	\frac{G(k)}{(k)^{-iy}} }{l^{x}} \Bigg] \\ 
	& = \Bigg[ \sum_{p=1}^{\infty} \frac{F(p)G(p)}{p^{2x}} \Bigg] \Bigg[ \sum_{l=1}^{\infty}\frac{ \frac{F}{\text{Id}_{e}^{iy}} \square \frac{G}{\text{Id}_{e}^{-iy}} (l) }{l^{x}} \Bigg] \\ 
	& = D(F \times G,2x) \times D( \frac{F}{\text{Id}_{e}^{iy}} \square \frac{G}{\text{Id}_{e}^{-iy}} ,x) \\
	& = \Bigg[ \sum_{p=1}^{\infty} \frac{F(p)G(p)}{p^{2x}} \Bigg] \Bigg[ \sum_{l=1}^{\infty}\frac{ F ~ \times ~ \text{Id}_{e}^{-iy} \square G ~ \times ~ \text{Id}_{e}^{iy} (l) }{l^{x}} \Bigg] \\ 
	\end{align*}
\end{proof}

\orthproduct*

\begin{proof}
	\[ \forall F,G \in \mathbb{M_{c}} ~: ~  F \times G = \delta_{1} \Rightarrow D(F \times G,s) = 1 \]
	so :
	\[ \forall F,G \in \mathbb{M_{c}} ~: ~  D(F,s) \times D(G,s) = D( F \square G ,s) \times D(F \times G,s) = D( F \square G ,s)  \]
\end{proof}

\primecompfactor*

\begin{proof}
	\textnormal{Let $ A \in \mathbb{P} $ and $ \bar{A} \in \mathbb{P} $ be complementary in  $ \mathbb{P} $  , let F an arithmetic completly multiplicative function : }
	\begin{align*}
	D(\mathds{1}_{A} \times F,s) \times D(\mathds{1}_{\bar{A}} \times F,s) 
	&= D([\mathds{1}_{A} \times F ] \times [ \mathds{1}_{\bar{A}} \times F ] ,s) \times D([ \mathds{1}_{A} \times F] \scalebox{0.5}{$\square$} [ \mathds{1}_{\bar{A}} \times F],s) \\
	&= D( \mathds{1}_{A} \times F \times  \mathds{1}_{\bar{A}} \times F,s) \times D([ \mathds{1}_{A} \times F] \scalebox{0.5}{$\square$} [ \mathds{1}_{\bar{A}} \times F],s)\\
	&= D( \mathds{1}_{A} \times  \mathds{1}_{\bar{A}} \times F  \times F ,s) \times D([ \mathds{1}_{A} \scalebox{0.5}{$\square$}  \mathds{1}_{\bar{A}}] \times F,s)\\
	&= D( \mathds{1}_{\emptyset} \times F  \times F ,s) \times D(\mathds{1} \times F,s) \\
	&= D( F,s)
	\end{align*}

\end{proof}

\twotime*

\begin{proof}
	\[[F ~ \square ~ F ] (m) = \prod_{p|m} F({p}^{v_{p}(m)})+F({p}^{v_{p}(m)}) = [{2}^{\omega(.)} \times F ](m) \]
\end{proof}

\begin{restatable}{prop}{ideplusone}\label{prop:ideplusone}

	\textnormal{Identity in $ (\mathbb{M},  \square  , \times )  $ : }
	\[\scalebox{1}{$ \text{Id}_{e} $}  \square \mathds{1} (m) = \hat{\sigma}(m)\]
	\textnormal{here the sigma is the sum of the prime divisors with there complementary. }

\end{restatable}

\begin{proof}

	\[\scalebox{1}{$ \text{Id}_{e} $}  \square \mathds{1} (m) = \prod_{p|m} p^{v_{p}(m)} + 1 = \sum_{ab=m ~ ( a , b ) = 1} a = \hat{\sigma}(m)\]
\end{proof}

\begin{restatable}{prop}{eulerchar}\label{prop:eulerchar}
		\textnormal{Identity in $ (\mathbb{M},  \square  , \times )  $ : }
	\[\phi = \scalebox{1}{$ \text{Id}_{e} $} \times \Bigg[\mathds{1}\scalebox{0.6}{$\square$}\frac{(-1)^{\omega}}{\mathbf{rad}}\Bigg]\]
\end{restatable}

\begin{proof} 
	The rewriting of the function $ \phi $  are derived simply from the following formula :
	\[\phi(n) = n \times \prod_{i=1}^{\omega(n)}{\left(1-\frac1{p_i}\right)} \]
\end{proof}

\sumchar*

\begin{proof}
\begin{align*}
	 \forall n \in \mathbb{N}-\{0,1\} ~ \forall a \in \mathbb{N}-\{0,1\}  : \\ \bigg[\stackrel{\phi(n)} {  \mathlarger{ \underset{k=1}{  \mathlarger{ \mathlarger{ \square } }  } }   }\chi_{k}  \bigg] (a)
	 &=\prod_{p|a} \sum_{k=1}^{\phi(n)} \chi_{k}({p}^{v_{p}(a)}) \\
	 &=\prod_{p|a} \phi(n) \cdot \delta_{{p}^{v_{p}(a)} \equiv 1 \mod n} \\
	 &=\prod_{p|a} \phi(n) \cdot \delta_{{p}^{v_{p}(a)}-1 \equiv 0 \mod n} \\
	 &=\phi(n)^{\omega(a)} \cdot \delta_{[\prod_{p|a} {p}^{v_{p}(a)}-1] \equiv 0 \mod n} \\
	 &=\phi(n)^{\omega(a)} \cdot \delta_{[ \text{Id}_{e} -\mathds{1}](a) \equiv 0 \mod n} \\
	 &=\phi(n)^{\omega(a)} \cdot \delta_{n | [  \text{Id}_{e} - \mathds{1}](a) } \\
\end{align*}
\end{proof}

\zetaminusone*

\begin{proof}
	\begin{align*}
		\sum_{n=1}^{\infty} \frac{1}{n^{s}} \cdot \sum_{p \in \mathbb{P}} \frac{1}{p^{s}} 
		 \frac{1}{\omega(np)} 
		&= \sum_{n=1}^{\infty} \sum_{p \in \mathbb{P}}  \frac{1}{p^{s}} \cdot \frac{1}{n^{s}} \cdot \frac{1}{\omega(np)}  \\
		&=\sum_{m=2}^{\infty} \sum_{\underset{p \in \mathbb{P} ~ n \in \mathbb{N^\star}}{pn=m} ~ }\frac{1}{np^{s}} \cdot \frac{1}{\omega(np)} \\
		&=\sum_{m=2}^{\infty} \sum_{\underset{p \in \mathbb{P} ~ n \in \mathbb{N^\star}}{pn=m} ~ }\frac{1}{m^{s}} \cdot \frac{1}{\omega(m)} \\
		&=\sum_{m=2}^{\infty} \frac{1}{m^{s}} \cdot \frac{1}{\omega(m)} \sum_{\underset{p \in \mathbb{P} ~ n \in \mathbb{N^\star}}{pn=m} ~ } 1 \\
		&=\sum_{m=2}^{\infty} \frac{1}{m^{s}} \cdot \frac{1}{\omega(m)} \cdot \omega(m) \\
		&=\sum_{m=2}^{\infty} \frac{1}{m^{s}} \\
		&=\zeta(s)-1 
	\end{align*}
	
\end{proof}

\zetaminusonenext*

\begin{proof}
	\begin{align*}
		\zeta(s)-1
		&= \sum_{n=1}^{\infty} \frac{1}{n^{s}} \cdot \sum_{p \in \mathbb{P}} \frac{1}{p^{s}} 
		\frac{1}{\omega(np)}\\
		&= \sum_{n=1}^{\infty} \sum_{p | n} \frac{1}{p^{s}} \cdot  \frac{1}{\omega(n)} \cdot \frac{1}{n^{s}} + \frac{1}{\omega(n)+1} \cdot \sum_{p \nmid n} \frac{1}{p^{s}} \cdot \frac{1}{n^{s}} \\ 
		&= \sum_{n=1}^{\infty} \frac{1}{n^{s}} \cdot [\sum_{p | n} \frac{1}{p^{s}} \cdot \frac{1}{\omega(n)} + \frac{1}{\omega(n)+1} \cdot \sum_{p \in \mathbb{P}} \frac{1}{p^{s}}- \sum_{p | n} \frac{1}{p^{s}} \cdot  \frac{1}{\omega(n)+1}  ]  \\ 
		&= \sum_{n=1}^{\infty} \frac{1}{n^{s}} \cdot [ \frac{1}{\omega(n)+1} \cdot \sum_{p \in \mathbb{P}} \frac{1}{p^{s}} + \sum_{p | n} \frac{1}{p^{s}}  \cdot [ \frac{1}{\omega(n)}  - \frac{1}{\omega(n)+1} ] ] \\ 
		&= \sum_{n=1}^{\infty} \frac{1}{n^{s}} \cdot [ \frac{1}{\omega(n)+1} \cdot \sum_{p \in \mathbb{P}} \frac{1}{p^{s}} + \sum_{p | n} \frac{1}{p^{s}}  \cdot [ \frac{1}{\omega(n)} \cdot \frac{1}{\omega(n)+1} ] ] \\ 
		&= \sum_{n=1}^{\infty} \frac{1}{n^{s}} \cdot [ \frac{1}{\omega(n)+1} \cdot \sum_{p \in \mathbb{P}} \frac{1}{p^{s}} + \frac{1}{\omega(n)+1} \cdot \sum_{p | n} \frac{1}{p^{s}}  \cdot  \frac{1}{\omega(n)} ] \\ 
		&= \sum_{n=1}^{\infty} \frac{1}{n^{s}} \cdot \frac{1}{\omega(n)+1} \cdot [ \sum_{p \in \mathbb{P}} \frac{1}{p^{s}} +  \frac{1}{\omega(n)} \cdot \sum_{p | n} \frac{1}{p^{s}}  ]  
	\end{align*}
\end{proof}

\section{Derivation over the ring $ ( \mathds{M} , \square , \times ) $  }
Using the previously constructed ring $ (\mathbb{M},  \scalebox{0.6}{$\square$}  , \times ) $ , we will in this section prove a result of the derivation operators that operate on this space

\begin{dfn}
	 An element $ F \in \mathbb{M} $ is said to be idempotent if :
	 \[ F \times F = F \]
	 
\end{dfn}

\begin{dfn}
	 A Dirichlet principle character $ \chi $ is a Dirichlet character of period $ k \in \mathbb{N}^{\star} $ with the following property :
\[
	 \chi(n) =
	\begin{cases}
	\begin{array}{ll}
	1 \,\textnormal{ if } \, (n ,  k) = 1 \\
	0 \,\textnormal{ if not}
	\end{array}
	\end{cases} 
\]

\end{dfn}
\begin{prop}

\textnormal{A Dirichlet principle character is always idempotent :}
\[ \chi \times \chi = \chi \]

\end{prop}

\begin{proof}
A dirichlet character is a completly multiplicative function , so it is characterised by it's value over the primes .
As the possible value of any principle dirichelet character are 0 or 1 , we obtain the idempotence property 

\[ \forall p \in \mathbb{P} ~,~ \chi(p)^{2} =  \chi(p)
 \] 
which give us it's idempotent identity :
\[ \chi \times \chi = \chi \]

\end{proof}

\begin{prop}
	\textnormal{For all modulus} $ k \in \mathbb{N} $ \textnormal{, all dirichlet character set to the power} $ \phi(n) $ \textnormal{are equal to a principle dirichlet character :}
\[
	 \chi(n)^{\phi(k)} =
	\begin{cases}
	\begin{array}{ll}
	1 \,\textnormal{ if } \, (n ,  k) = 1 \\
	0 \,\textnormal{ if not}
	\end{array}
	\end{cases} 
\]

\end{prop}

\begin{proof}
Using the Fermat–Euler theorem , we have :
\[ (n,k) = 1 \Rightarrow n^{\phi(k)} \equiv 1 \textnormal{(mod k)} \]
Which allow us to compute :
\[
	\chi(n^{\phi(k)})  = \chi(1+mk)
	 = \chi(1+mk) =
	\begin{cases}
	\begin{array}{ll}
	1 \,\textnormal{ if } \, (n ,  k) = 1 \\
	0 \,\textnormal{ if not}
	\end{array}
	\end{cases} 
\]
Using the completely multiplicative property of $ \chi $ as a dirichlet character :
\[
	\chi(n^{\phi(k)})  = \chi(n)^{\phi(k)}
\]
So we have :
\[ \chi(n)^{\phi(k)} = \chi(n^{\phi(k)}) = 			\begin{cases}
	\begin{array}{ll}
	1 \,\textnormal{ if } \, (n ,  k) = 1 \\
	0 \,\textnormal{ if not}
	\end{array}
	\end{cases}  \]
\end{proof}

\begin{prop}
\textnormal{Any derivation acting on the ring} $ (\mathbb{M},  \scalebox{0.6}{$\square$}  , \times ) $ \textnormal{is null over all principal dirichlet characters :}
	
	\[ \chi \textnormal{~ is principal} \Rightarrow \nabla(\chi) = \mathds{1}_{\varnothing}  \]
\end{prop}

\begin{rem}
For notation purpuses , we will write the function $ \mathds{1}_{\varnothing} $ as $ \numbermathbb{0} $ in this section
\end{rem}

\begin{proof}
Let's take a principale dirichlet character $ \chi $ , by the derivation operator , we have :
	\[ \nabla(\chi^{2})  =  \numbermathbb{2} \times \chi \times \nabla(\chi) \]
The principal dirichlet character are idempotent , so we have the following identity
	\[ \nabla(\chi^{2})  =  \nabla(\chi) \]
Which gives us
	\[ \nabla(\chi) = \numbermathbb{2} \times \chi \times \nabla(\chi) = \numbermathbb{0} \]
	\[ \Rightarrow [ \numbermathbb{2} \times \chi - \numbermathbb{1} ] \times \nabla(\chi) = \numbermathbb{0}  \]
	Knowing that a multiplicative function is characterised by it's value over prime powers , we have :
	 \[ \forall p \in \mathbb{P} , \forall n \in \mathbb{N}^{\star} ~,~  \chi(p^{v_{p}(n)}) = 0 \textnormal{ or } \chi(p^{v_{p}(n)}) = 1 \Rightarrow \nabla(\chi)(p^{v_{p}(n)})  = 0 \]

Wich gives us that for all principal dirichlet characters :
\[ \nabla(\chi) = \numbermathbb{0}  \]
\end{proof}

\derivation*

\begin{proof}
Let's define the set of zeros of a dirichlet character 
\[ Z_{\chi} = \left\{ a^{n} | n \in \mathbb{N}^{\star} ,  a \in \mathbb{P} , \chi(a)=0 \right\} \]
over which we could build an indicatrice function $ \mathds{1}_{ Z_{\chi} } $ . We have constructed the set $  Z_{\chi} $ to fill the zeros , and only the zeros , of the dirichlet character $ \chi $ , so we have  : 
\[ \forall n \in \mathbb{N}^{\star}, \forall p \in \mathbb{P} ~,~  [ \chi(p^{n}) = 0  \textnormal{ and } \mathds{1}_{ Z_{\chi} }(p^{n}) \ne 0 ] \textnormal{ or } [ \chi  (p^{n}) \ne 0  \textnormal{ and } \mathds{1}_{ Z_{\chi} }(p^{n}) = 0 ] \]
note that this or is exclusive , which gives us 
\[ \forall n \in \mathbb{N}^{\star}, \forall p \in \mathbb{P} ~,~ ( \chi \scalebox{0.6}{$\square$} \mathds{1}_{ Z_{\chi} } )(p^{n}) \ne 0   \]
and
\[  \chi \times \mathds{1}_{ Z_{\chi} } =\numbermathbb{0} \]
by the derivation operator rules, we have :
\begin{align*}
\nabla( \chi \scalebox{0.6}{$\square$} \mathds{1}_{ Z_{\chi} } \scalebox{0.6}{$\square$} (-\mathds{1}_{ Z_{\chi} }) )  = \nabla( \chi \scalebox{0.6}{$\square$} \mathds{1}_{ Z_{\chi} } ) \scalebox{0.6}{$\square$} \nabla(-\mathds{1}_{ Z_{\chi} })
\end{align*}
\begin{align*}
 \forall n \in \mathbb{N}^{\star}~~ \forall p \in \mathbb{P} ~,~ ( \chi \scalebox{0.6}{$\square$} \mathds{1}_{ Z_{\chi} } )^{\phi(n)} 
 &  = \stackrel{\phi(n)} { \mathlarger{ \mathlarger{ \underset{k=1}{ ~  \mathlarger{ \mathlarger{ \square } }  ~ } }}  }  \binom{\phi(n)}{k} \times \chi^{k} \times \mathds{1}_{ Z_{\chi} }^{\phi(n)-k} \\
 &  = \chi^{\phi(n)}  \scalebox{0.6}{$\square$}   \mathds{1}_{ Z_{\chi} }^{\phi(n)} \\
 &  = \chi  \scalebox{0.6}{$\square$}   \mathds{1}_{ Z_{\chi} }
\end{align*}

so the element $ ( \chi \scalebox{0.6}{$\square$} \mathds{1}_{ Z_{\chi} } )^{\phi(n)}  $ is an idempotent element

we could compute the derivation of this idempotent element in two way :
\[ \nabla( ( \chi \scalebox{0.6}{$\square$} \mathds{1}_{ Z_{\chi} } )^{\phi(n)} ) = \numbermathbb{0} \]	
\[ \nabla( \chi \scalebox{0.6}{$\square$} \mathds{1}_{ Z_{\chi} } )^{\phi(n)} ) = \phi(n) \times (\chi \scalebox{0.6}{$\square$} \mathds{1}_{ Z_{\chi} } )^{\phi(n)-1} \times \nabla (\chi \scalebox{0.6}{$\square$} \mathds{1}_{ Z_{\chi} } ) \]	
combining the two equaltions we obtain 
\[ (\chi \scalebox{0.6}{$\square$} \mathds{1}_{ Z_{\chi} } ) \times \nabla (\chi \scalebox{0.6}{$\square$} \mathds{1}_{ Z_{\chi} } ) =  \numbermathbb{0} \]	
we have also that 
\[ (\chi \scalebox{0.6}{$\square$} \mathds{1}_{ Z_{\chi} } ) (p^{n}) \ne 0 \]	
which give us 
\[ \nabla (\chi \scalebox{0.6}{$\square$} \mathds{1}_{ Z_{\chi} } ) =  \numbermathbb{0} \]	
we also know that $ \mathds{1}_{ Z_{\chi} } $ is an idempotent element , so we have 
\[ \nabla ( \mathds{1}_{ Z_{\chi} } ) = \numbermathbb{0} \]
which allow us to simplifie the equation 
\[ \nabla (\chi \scalebox{0.6}{$\square$} \mathds{1}_{ Z_{\chi} } ) = \nabla (\chi) \scalebox{0.6}{$\square$} \nabla(\mathds{1}_{ Z_{\chi} } ) = \nabla (\chi) =\numbermathbb{0} \]	
so we have 
\[ \nabla (\chi) =\numbermathbb{0} \]

\end{proof}

\section{Conclusion}

In this work, we have explored a possible factorization of some Dirichlet series and how to generalize it by constructing the ring $  (\mathbb{M}, \square, \times) $ of multiplicative functions.
We also proved that we can not get a similar ring by changing the weight of the convolution $ \square $ used, the indicator of the pairs of coprime integers is the only weight that generates this ring. 
$ \\ $
Several previous works have studied $\square $ convolution as a product \cite{AlkanZaharescuandZaki} \cite{SchwabSilberberg} \cite{SchwabSilberberg2} \cite{Yokom} \cite{CashwellEverett} . It is even called the unitary product in some references, but we use it in this article as a sum in the ring $  (\mathbb{M}, \square, \times) $ which has led us to different structures.
$ \\ $
$ \\ $
This allowed us to have:
\begin{itemize}
    \item A factorization of the Dirichlet series in some cases
    \item The zeta function being a special case of the dirichlet series, a generalization of a Hardy formula has been established
    \item Several functional identities internal to the ring have been found
       \item A general result on derivations and Dirichlet characters
\end{itemize}

\section{Appendix}

In this section we will build a vector space based on the ring $ (\mathbb{M},  \scalebox{0.6}{$\square$}  , \circ ) $ that could help us better understand some identities
\begin{dfn}
	We define the operation $ \circ $ as an external operation :
	\[
	\forall F \in \mathbb{M} , \forall \lambda \in \mathbb{C} , \forall n \in \mathbb{N}^{\star} ~~
	( \lambda \circ F ) (n) = \prod_{p | n} \lambda \cdot F(p^{v_{p}(n)})
	\]

\end{dfn}

\begin{prop}
$(\mathbb{M},  \scalebox{0.6}{$\square$}  , \circ ) ~ \, \textnormal{is a }\mathbb{C} \textnormal{-vectorial space }  $
\end{prop}

\begin{proof}
We have that $ (\mathbb{M},  \scalebox{0.6}{$\square$} ) $ is a commutative group , and the operation $ \circ  $ verifies the following assertions :
\[  
\forall F,G \in \mathbb{M} , \forall \lambda, \mu  \in \mathbb{C} ~ : ~ 
\begin{array}{l rcl}
\lambda \circ  (F \scalebox{0.6}{$\square$} G) 	= 	(\lambda \circ  F) \scalebox{0.6}{$\square$} (\lambda \circ  G)  \\
(\lambda \scalebox{0.6}{$\square$} \mu) \circ  F 	= 	(\lambda \circ  F) \scalebox{0.6}{$\square$} (\mu \circ  F) \\
(\lambda\mu) \circ  F 	= 	\lambda \circ  (\mu \circ  F) \\
1 \circ  F 	= 	F \\
\end{array}
\]
which give that  $ (\mathbb{M},  \scalebox{0.6}{$\square$}  , \circ ) ~ \, \textnormal{is a }\mathbb{C} \textnormal{-vectorial space } $
\end{proof}

\begin{prop}
$(\mathbb{M},  \scalebox{0.6}{$\square$}  , \circ ) $ \textnormal{is also} $\mathbb{R} $ \textnormal{-vectorial space by taking the same definition of the external operation over the field} $ \mathbb{R}$ \\

\end{prop}

\begin{proof}
The proof on $ \mathbb{C} $ hold for  $ \mathbb{R} $ by modifying the set of the external operation 
\end{proof}

\begin{rem}

$(\mathbb{M},  \scalebox{0.6}{$\square$}  , \circ )  $ \textnormal{over} $\mathbb{R}$ \textnormal{is a  vector subspace of} $(\mathbb{M},  \scalebox{0.6}{$\square$}  , \circ ) $ \textnormal{over} $\mathbb{C}$

\end{rem}

\begin{thm}

$ \mathbb{M}_{\mathbb{C}} = \mathbb{M}_{\mathbb{R}} \oplus i \mathbb{M}_{\mathbb{R}} $ 

\end{thm}

\begin{rem}

This could be usefull to project element over the vector subspace $ \mathbb{M}_{\mathbb{R}} $.
\end{rem}


\bibliographystyle{amsplain}
\bibliography{references/references}

\end{document}